\documentclass[a4paper,12pt]{amsart}
\usepackage[utf8]{inputenc}
\usepackage{amsmath, amsfonts, amsthm}

\title{The monochromatic stable Hopf invariant}
\author{Guozhen Wang}
\address{Department of Mathematics, MIT, 
Cambridge, MA 02139}
\email{guozhen@math.mit.edu}
\urladdr{www.math.mit.edu/$\sim$guozhen}

\newtheorem{thm}{Theorem}[section]
\newtheorem{cor}[thm]{Corollary}
\newtheorem{lem}[thm]{Lemma}
\newtheorem{ex}[thm]{Example}

\newtheorem{defi}{Definition}[section]
\newtheorem{rem}{Remark}[section]

\begin{document}

\begin{abstract}
 In this paper we will compute the effect of the James-Hopf map after applying the Bousfield-Kuhn functor on Morava $E$-theory, and then compute the monochromatic Hopf invariant of the $\beta$ family using this cohomological information.
\end{abstract}

\maketitle

\section{Introduction}

We will study the unstable homotopy of odd spheres from the chromatic point of view. The chromatic zero case is simple: the rational homotopy of odd spheres are the same as the stable one. So the interesting problem begins with chromatic height one.

The chromatic one component in the stable homotopy groups are the image of $J$. The behavior of these elements are studied in \cite{ma82}, \cite{gr84} and \cite{mt92}. In particular, it was shown by Mahowald and Thompson that, at an odd prime $p$, the $v_1$-inverted EHP sequence has a quite regular pattern. In this case when $v_1$ is inverted, there is an isomorphism from EHP spectral sequence to the Atiyah-Hirzebruch spectral sequence for the stable homotopy groups of $B\Sigma_p$, the classifying space of the symmetric group on $p$ objects. Furthermore, using the $K(1)$-local isomorphism of the sphere spectrum and $\Sigma^\infty B\Sigma_p$, we can identify the Atiyah-Hirzebruch spectral sequence with the $p$-Bockstein spectral sequence of the $K(1)$-local sphere, whose structure is well-known.

This paper is the first one in a series to understand the pattern of the EHP sequence at chromatic height two. In general, we have the Goodwillie tower of the identity functor extending the comparison of the unstable spheres and stable structure of $B\Sigma_p$ in the $K(1)$-local case. It was proved by Arone and Mahowald that at finite chromatic height the Goodwillie tower becomes finite. In particular, if we look at chromatic height two, the tower has only three layers. So it is reasonable to have a complete understanding of the whole tower, which can be viewed as a three term sequence of spectra.

This paper is concerned with understanding the first map in the sequence, which is the $K(2)$-localized James-Hopf map. We will describe its effect on $E_2$-cohomology, which will give information on the homotopy groups via Adams-Novikov spectral sequence. As an application of the formula, we will compute the second monochromatic stable Hopf invariant of the $\beta$ family. From the computations of Bendersky and others, it was suggested that the Hopf invariant of $\beta_{i/j,k}$ should be $\beta_{i-j/k}$. In this paper, we will show that, when a $\beta$ element does desuspend as expected, then the stable Hopf invariant is as suggested; otherwise its stable Hopf invariant has chromatic height at least three.

\subsection{Backgrounds and notations}

\subsubsection{Morava $E$-theory and monochromatic layers}

Fix a prime $p$. $K(n)$ denotes the Morava $K$-theory of height $n$. $E_n$ denotes the Morava $E$-theory, with homotopy group $\mathbb{Z}_p[[u_1,\dots,u_{n-1}]][u^{\pm}]$. We have $v_i=u_iu^{p^i-1}$ for $i=1,\dots,n-1$, and $v_n=u^{p^n-1}$.
We will denote by $F_n$ the formal group associated to $E_n$, and $+_{F_n}$ will mean the multiplication according to the formal group law $F_n$.

Let $L_{K(n)}$ be the localization functor with respect to Morava $K$-theory. We denote the functor of taking the $n^{th}$ monochromatic layer by $M_n$, which makes the following sequence a distinguished triangle: $M_nX\rightarrow L_{K(n)}X\rightarrow L_{K(n-1)}L_{K(n)}X$.

The functors $M_n$ and $L_{K(n)}$ factors through each other. We have $M_n=M_nL_{K(n)}$ and $L_{K(n)}=L_{K(n)}M_n$. The Morava $E$-theory of monochromatic layers of the sphere is computed in \cite{rav84}. We have $${E_n}_*M_n\mathbb{S}={E_n}_*/(p^\infty,\dots,v_{n-1}^\infty)$$
Hence we have $${E_n}_*M_nM(p)={E_n}_*/(p,v_1^\infty,\dots,v_{n-1}^\infty)$$ for the Moore spectrum $M(p)$, and in general $${E_n}_*M_nV(k)={E_n}_*/(p,\dots,v_k,v_{k+1}^\infty,\dots,v_{n-1}^\infty)$$ for the Smith-Toda complex $V(k)$.

\subsubsection{Periodic unstable homotopy and Bousfield-Kuhn functor}

We are concerned with unstable periodic homotopy in this paper. For any spaces $X$ and $W$ with base point, we denote by $\pi_i(X;W)$ to be the homotopy group of $Map_*(W,X)$, the mapping space from $W$ to $X$ preserving the base point. Suppose $W$ is a finite CW complex of type $n$, which means that $K(h)_*W=0$ for $h<n$ and $K(n)_*W$ is nontrivial. Then by the Hopkins-Smith periodicity theorem, for some suspension of $W$ there is a $v_n$-self map, i.e. there is a map $$v_n^t:\Sigma^{N+t|v_n|}W\rightarrow \Sigma^NW$$ for certain $t$ and $N$, which induce multiplication by $v_n^t$ in $K(n)_*W$. This map induces the multiplication by $v_n^t$ map on $\pi_i(X;W)$ for $i$ large enough. 

\begin{defi}
The $v_n$-periodic homotopy groups of $X$ with coefficients $W$ is defined to be $(v_n^t)^{-1}\pi_*(X;W)$. We will denote it by $v_n^{-1}\pi_*(X;W)$.
\end{defi}
Note that, since $v_n^t$ has positive degree, lower dimensional items do not affect the result of inverting $v_n$.
One can see that this notion only depends on (functorially) the suspension spectrum of $W$ when $X$ is fixed.

Now we let $$W_0\rightarrow W_1\rightarrow\cdots$$ to be a system of type $n$ spectra with colimit $\mathbb{S}/(p^{\infty},\dots,v_{n-1}^{\infty})$ with the top cell in dimension $0$. Then we can form an inverse system of groups $$v_n^{-1}\pi_*(X;W_0)\leftarrow v_n^{-1}\pi_*(X;W_1)\leftarrow\cdots$$

\begin{defi}
The $v_n$-periodic homotopy groups of $X$ are defined to be the inverse limit $\lim v_n^{-1}\pi_*(X;W_i)$. We denote it by $v_n^{-1}\pi_*(X)$.
\end{defi}
One can show that this notion depends only on $X$.

The Bousfield-Kuhn functor enables one to investigate the periodic homotopy groups using stable homotopy theory.
We will give a brief summary of its properties. See \cite{kuhn08} for details.

Let $W$ be a type $n$ complex. 
The $v_n$-self map induces a map $$Map_*(\Sigma^NW,X)\rightarrow Map_*(\Sigma^{N+t|v_n|}W,X)=\Omega^{t|v_n|}Map_*(\Sigma^NW,X)$$
Using this map, we form the pre-spectrum with $Map_*(\Sigma^NW,X)$ sitting in every place which are congruent to $-N$ mod $t|v_n|$, and the structure maps as above.
\begin{defi}
Define $\Phi_n^WX$, the Bousfield-Kuhn functor with coefficient $W$
to be the associated spectrum. 
\end{defi}
One can show that this functor only depends on the suspension spectrum of $W$.
As before we can take a system $W_0\rightarrow W_1\rightarrow\dots$ of type $n$ spectra tending to $\mathbb{S}/(p^{\infty},\dots,v_{n-1}^{\infty})$.
\begin{defi}
Define the Bousfield-Kuhn functor $\Phi_nX$ to be the inverse limit $\lim\Phi_n^{W_i}X$.
\end{defi}

One can see that we have $\pi_*\Phi_nX=v_n^{-1}\pi_*(X)$, when the $lim^1$ terms vanish, which holds for spheres.

We have variants of the Bousfield-Kuhn functor.
\begin{defi}
Define the $K(n)$-local Bousfield-Kuhn functor $\Phi_{K(n)}$ to be $L_{K(n)}\Phi_n$. Define the monochromatic Bousfield-Kuhn functor $\Phi_{M_n}$ to be $M_n\Phi_n$.
\end{defi}

From the definition, we have $\Phi_{K(n)}\Omega^\infty Y=L_{K(n)}Y$ for any spectrum $Y$.

\subsubsection{Monochromatic Hopf invariants}
We will investigate the Hopf invariant in this paper. The Hopf invariant is one of the three maps in the EHP sequence.  See Section 1.5 of \cite{rav86} for details of the EHP sequence.

In general, the Hopf invariant is defined for any space $X$ to be the adjoint of the map $\Sigma\Omega\Sigma X\rightarrow \Sigma X^{\wedge p}$, which is the projection map using Snaith splitting $\Sigma\Omega\Sigma X=\vee_i \Sigma X^{\wedge i}$, introduced in \cite{snaith74}.

We will study its stable analogue, which is the James-Hopf map $$jh:\Omega^{\infty}\Sigma^{\infty}X\rightarrow \Omega^{\infty}(\Sigma^{\infty}X^{\wedge p})_{h\Sigma_p}$$ 
defined to be the adjoint of the projection map $$\Sigma^{\infty}\Omega^{\infty}\Sigma^{\infty}X\rightarrow (\Sigma^{\infty}X^{\wedge p})_{h\Sigma_p}$$ using the Snaith splitting of $\Sigma^{\infty}\Omega^{\infty}\Sigma^{\infty}X$.

We will only look at the case when $X$ is a sphere. When specialized to the case $X=S^0$, we have the James-Hopf map $$jh:\Omega^{\infty}\Sigma^{\infty} S^0\rightarrow \Omega^{\infty}\Sigma^{\infty} {B\Sigma_p}_+$$
It is well-known that $\Sigma^{\infty} {B\Sigma_p}_+$ can be realized as a CW spectrum with cells in dimensions $0, 2(p-1)-1,2(p-1),4(p-1)-1,4(p-1),\dots$. We will denote by $P_k$ to be the quotient spectrum obtained from it by collapsing the first $k$ cells. Then the James-Hopf map for $S^k$ becomes $\Omega^{\infty+k}\Sigma^{\infty}S^k\rightarrow \Omega^{\infty}P_k$, and the suspension map on the left side is compatible with the quotient map on the right side. Note that $P_1=\Sigma^\infty
B\Sigma_p$.

The skeletal filtration on $P_1$ gives an Atiyah-Hirzebruch type spectral sequence: $$\oplus \pi_*(\mathbb{S})\Rightarrow \pi_*(P_1)$$ 
So any element in $\pi_*(P_1)$ has a representative in the $E_1$-term of this spectral sequence. Starting with an element in $\pi_*(\mathbb{S})$, the James-Hopf map sends it into an element in $\pi_*(P_1)$. We will call the representative in the $E_1$-term of the AHSS of the image to be the stable Hopf invariant. In other words, for any $\xi\in \pi_*(\mathbb{S})$, there is an largest $k$ when the image of $jh_*(\xi)$ under the map $P_1\rightarrow P_k$ is non-zero. Then the map $jh_*(\xi):\Sigma^{|\xi|}\mathbb{S}\rightarrow P_k$ can be lifted to a map $\Sigma^{|\xi|}\mathbb{S}\rightarrow \Sigma^{s_k}\mathbb{S}\rightarrow P_k$, where the last map is the inclusion of the bottom cell. Then the stable Hopf invariant of $\xi$ is the homotopy class represented by the first map.
Using the compatibility of the suspension and the filtration on $P_1$, one can see that the stable Hopf invariant gives a lower bound on the sphere of origin for stable elements, and when the lower bound is achieved it detects the Hopf invariant.

This paper is concerned with the monochromatic stable Hopf invariant. We will apply the Bousfield-Kuhn functor $\Phi_{M_n}$ to the James-Hopf map $jh:\Omega^{\infty}\mathbb{S}\rightarrow \Omega^{\infty}P_1$ to get the monochromatic James-Hopf map $$\Phi_{M_n}jh:M_n\mathbb{S}\rightarrow M_nP_1$$ 
We know $M_nP_1$ is the direct limit of its skeleton, so we have the monochromatic analogue of the AHSS computing the monochromatic homotopy groups: $$\oplus \pi_*(M_n\mathbb{S})\Rightarrow \pi_*(M_nP_1)$$ 
So taking the representative in the $E_1$-term of the image under the monochromatic James-Hopf map gives the monochromatic stable Hopf invariant. Hence we have the diagram $$\begin{array}{ccc}\pi_*M_n\mathbb{S}&\rightarrow& \pi_*M_n P_1 \\
\Downarrow&&\downarrow\\\pi_*M_n\Sigma^{s_k}\mathbb{S}&\rightarrow&\pi_*M_nP_k
\end{array}$$ 
with the upper arrow the James-Hopf map, the lower arrow the inclusion of the bottom cell, and the lifting problem for the left arrow gives the monochromatic Hopf invariant.

\subsubsection{Morava $E$-theory of some classifying spaces} \label{ebg}

We will need the Morava $E$-theory of $B\Sigma_p$, the classifying space of the symmetric group. The details of the computations  can be found in \cite{str98}.

Recall that, using the fibration $B\mathbb{Z}/p\rightarrow \mathbb{C}P^\infty\xrightarrow{p}\mathbb{C}P^\infty$ we can identify $E_n^*B\mathbb{Z}/p$ with $E_n[\xi]/[p](\xi)$, with $[p](\xi)$ the $p$-series of the formal group associated to $E_n$.

The $p$-Sylow subgroup of $\Sigma_p$ is the subgroup $\mathbb{Z}/p$ generated by the cyclic permutation. Its Weyl group can be identified with $\mathbb{F}_p^\times$, which is a cyclic group of order $p-1$. The restriction map $E_n^*B\Sigma_p\rightarrow E_n^*B\mathbb{Z}/p$ lands in the invariants of $\mathbb{F}_p^\times$. It turns out that we have the isomorphism $E_n^*B\Sigma_p=(E_n^*B\mathbb{Z}/p)^{\mathbb{F}_p^\times}$. The generator $\xi$ lies in cohomological degree $2$, and has weight $1$ for the action of $\mathbb{F}_p^\times$. So we have $E_n^*B\Sigma_p=E_n^*[x]/xq(x)$, where $x=\xi^{p-1}$ and $q(\xi^{p-1})=\frac{[p](\xi)}{\xi}$.
Consequently $E_n^*P_1=xE_n^*[x]/q(x)$.

\subsubsection{Other conventions} \label{ocv}

We will study several $\mathcal{E}_\infty$ and co-$\mathcal{E}_\infty$ algebras and coalgebras in this paper. To make the notion easier, we will work in the category of $S$-modules in the sense of \cite{ekmm}. By $\mathcal{E}_\infty$ and co-$\mathcal{E}_\infty$ we will simply mean commutative and co-commutative respectively in the category of $S$-modules. 
We will denote the free $\mathcal{E}_\infty$-algebra generated by a spectrum $X$ by $\mathcal{E}_\infty(X)$. The free commutative algebra generated by $X$ will be denoted $S(X)$, and it will be equivalent to $\mathcal{E}_\infty(X)$ in the model of spectra we will use.

We will introduce the (homotopy) category of weakly graded spectra. The objects are sequences of spectra $(X_i)_{i\in\mathbb{Z}}$. For morphisms, let $X=(X_i)$ and $Y=(Y_i)$ be two objects in the category. Then a map can be identified with a matrix $(f_{ij})$ for $f_{ij}:X_i\rightarrow Y_j$. Then this map is a morphism of weakly graded spectra, if for fixed $i$ there are only finitely many $j$ with $f_{ij}$ nontrivial, and for fixed $j$, there are only finitely many $i$ with $f_{ij}$ nontrivial.

We will use $D$ for Spanier-Whitehead dual. For a graded spectrum $(X_i)_{i\in\mathbb{Z}}$, we introduce the restricted Spanier-Whitehead dual functor, denoted by $D'$, to be the direct sum of the Spanier-Whitehead dual of its homogeneous pieces, i.e. the sequence $(Y_i)_{i\in\mathbb{Z}}$ with $Y_i=DX_{-i}$. One sees that the functor $D'$ can be defined on the category of weakly graded spectra.

We also introduce the (homotopy) category of weak infinite loop spaces. The objects are labeled by pointed connected spaces. For any two objects labeled by $X$ and $Y$, a morphism is a (homotopy class of) pointed map from $\Omega^\infty\Sigma^\infty X$ to $\Omega^\infty\Sigma^\infty Y$ such that the adjoint $\Sigma^\infty\Omega^\infty X\rightarrow \Sigma^\infty Y$ has only finitely many (homotopically) nontrivial components when the domain is decomposed using the Snaith splitting. There is a functor to the category of spaces by taking $X$ into $\Omega^\infty\Sigma^\infty X$, and the composition of this functor with infinite suspension lands in the category of weakly graded spectra.

The symbols $\prod$ and $\coprod$ will always mean the categorical product and coproduct respectively.

\subsection{Statement of the main results}\label{mr} 

We will be mainly concerned with the James-Hopf map for $S^1$ after applying the Bousfield-Kuhn functor, i.e. $$jh:L_{K(n)}\mathbb{S}\rightarrow L_{K(n)}\Sigma^\infty B\Sigma_p$$
First we will compute the effect of the James-Hopf map on the $E_n$-cohomology.
\begin{thm} \label{sfl}
The map on $E_n$-cohomology induced by the James-Hopf map $jh$ is as follows:
Let $x_1,x_2,\dots,x_{p^{n-1}+\dots+p+1}$ be the roots (over some extension of $E_n^*$) of the polynomial $q(x)$ (which is defined in Section \ref{ebg}). Then for any polynomial $f(x)$ in $x(E_n^*[x]/q(x))$,
we have $jh^*(f(x))$ equals $$\frac{1}{p}\sum_i{f(x_i)}$$ up to a factor which is a unit.
\end{thm}
\begin{rem}
The previous map is essentially the map $\frac{tr}{p}$, with $tr$ the trace map for the finite extension $E_n^*\rightarrow E_n^*[x]/q(x)$. Since the constant term in $q(x)$ is divisible by $p$, we always get integral results. 
\end{rem}
\begin{rem}
There is the dual description in \cite{rez13}. It is comparable with the previous one if we relate them by the Weil pairing. We will explain this relation in detail later.
\end{rem}

Recall that $\Sigma^\infty B\Sigma_p$ can be realized as a CW spectrum with cells in dimensions $2(p-1)-1, 2(p-1), 4(p-1)-1, 4(p-1), \dots$. We define an increasing filtration where the $k^{th}$ filtration constitutes the cells of dimension up to $2k(p-1)$. This amounts to a filtration on $\Sigma^\infty B\Sigma_p$, such that the $k^{th}$ sub-quotient is the suspended Moore spectrum $\Sigma^{k|v_1|-1}M(p)$. Using this filtration we get an Atiyah-Hirzebruch type spectral sequence $$\pi_t\Sigma^{k|v_i|-1}{E_n}_*M_nM(p)\Rightarrow\pi_{t+k|v_i|-1}({E_n}_*M_n\Sigma^\infty B\Sigma_p)$$
We will describe the cohomological Hopf invariant of the $\beta$-family by looking at the element detecting $jh_*(\frac{v_2^k}{p^iv_1^j})$ in the AHSS for $M_n\Sigma^\infty B\Sigma_p$.

\begin{thm}\label{hlhi}
Let $p$ be odd prime. Let $\frac{v_2^k}{p^{i+1}v_1^j}$, with $p^i|j$, be an element in the $E_2$-homology of the $2^{nd}$ monochromatic layer of $\mathbb{S}$. Then its image under $jh_*$ lies in the $(pj+i+1)^{st}$ filtration, detected by $\frac{v_2^{k-j}}{v_1^{i+1}}$ in AHSS as an element in the $E_2$-homology of the $2^{nd}$ monochromatic layer of the corresponding sub-quotient, which is a suspended Moore spectrum.
\end{thm}

To get the homotopical Hopf invariant, we need to deal with ANSS differentials. However, in the case of $p\geq5$, there are no ANSS differentials, so that we get the monochromatic Hopf invariants of the $\beta$-family, which fits the computations of Bendersky and others:
\begin{thm}\label{k2hi}
For $p\geq5$, the $2^{nd}$ monochromatic stable Hopf invariant of $\beta_{k/j,i}$ is $\beta_{k-j/i}$ supported on the top cell of the $(pj+i)^{th}$ sub-quotient Moore spectrum.
\end{thm}

\subsection{Outline of the proof}

Our goal is to compute the effect of the Bousfield-Kuhn functor on the James-Hopf map. We will use the TAQ interpretation of the Bousfield-Kuhn functor introduced in \cite{br12} to compute the Bousfield-Kuhn functor.

Let $X$ be a space. Then the Snaith splitting says that the suspension spectrum of $\Omega^\infty\Sigma^\infty X$ is the free $\mathcal{E}_\infty$ spectrum generated by the suspension spectrum of $X$. The diagonal map on $\Omega^\infty\Sigma^\infty X$ gives a co-$\mathcal{E}_\infty$ coproduct on $\Sigma^\infty\Omega^\infty\Sigma^\infty X$. Then from the formulas in \cite{kuhn00} we find that this coproduct induces an $\mathcal{E}_\infty$ product on its restricted Spanier-Whitehead dual $D'\Sigma^\infty\Omega^\infty\Sigma^\infty X$. From \cite{gk95} we know that the norm map gives a $K(n)$-local equivalence of homotopy orbits and homotopy invariants, together with the formula of the coproduct in \cite{kuhn00}, we conclude that $D'\Sigma^\infty\Omega^\infty\Sigma^\infty X$ is $K(n)$-locally equivalent to the $\mathcal{E}_\infty$ spectrum generated by $D\Sigma^\infty X$. However, the natural gradings are not preserved under this equivalence. From \cite{br12} we know that the Bousfield-Kuhn functor on $\Omega^\infty\Sigma^\infty X$ can be identified with taking the dual of the indecomposables in $D'\Sigma^\infty\Omega^\infty\Sigma^\infty X$, which means taking the first homogeneous piece under the grading induced from $\mathcal{E}_\infty DX$. Howerver, the usual way of describing the cohomology of $\Omega^\infty\Sigma^\infty X$, including the formula for the James-Hopf map on cohomology in \cite{kp85}, is using the grading induced by the Snaith splitting. So the problem of computing the monochromatic stable Hopf invariant on cohomology reduces to understand the relations of these two gradings so that we can use the formula in \cite{kp85}.

We will show that the comparison can be described using the product on $D\Sigma^\infty X$ and the norm map, which would lead to the formula in Theorem \ref{sfl}.

In order to compute the monochromatic Hopf invariant using Theorem \ref{sfl}, we will try to compute the leading term of the power sum of the roots of the $p$-series. The leading terms of the $p$-series is $v_2\xi^{p^2}+v_1\xi^p+p$. Using the Newton identity, the power sum of its roots would be a sum of terms of the form $\frac{k(s+t-1)!}{s!t!}p^sv_1^t$ with $k=(p+1)s+pt$. Naively we would expect the leading term should have no powers of $p$ introduced from the coefficient, which leads to the formula in Theorem \ref{hlhi}. We will show that this naive consideration does give the correct leading term.

With the cohomological computations in hand, Theorem \ref{k2hi} follows once we observe that there are no obstruction for large primes.

\subsection{Organization of the paper}

In Section \ref{bc}, we will discuss the structure of the suspension spectrum of $\Omega^\infty\Sigma^\infty X$, as an $\mathcal{E}_\infty$ and co-$\mathcal{E}_\infty$ bi-algebra, in detail. We will finally arrive at Theorem \ref{jhc}, which says that the James-Hopf map, after applying the Bousfield-Kuhn functor, is essentially the composition of the inverse of the norm map and the multiplication map for $DX$. 

Then in Section \ref{dc}, we will discuss the norm map, give a formula for the quadratic form it induces on homology, and prove Theorem \ref{sfl}.

With the formula, we will do the computations in Section \ref{mh} for the homological monochromatic Hopf invariant for the $\beta$ family, proving Theorem \ref{hlhi}. 

We will compare our formula with the modular isogeny complex using the Weil pairing in Section \ref{cmp}.

In Section \ref{hb}, we will show that for $p\geq5$, there are no obstructions preventing these homological computations to imply homotopical ones, proving Theorem \ref{k2hi}. 

In the last Section \ref{ah}, we will give a remark on how to get the actual Hopf invariant from the monochromatic information. The main point is that we can construct an unstable chromatic tower. The convergence issue of this tower is left open.

\subsection{Acknowledgments}

The author would like to thank Mark Behrens, Zhen Huan and Haynes Miller for discussions and suggestions. Especially Mark Behrens, who has review the paper and pointed out many places which can be improved.

\section{The behavior of the James-Hopf map on cohomology} \label{bc}

We will describe a method to determine the map on cohomology induced by the James-Hopf map after applying the Bousfield-Kuhn functor.

First recall that by Snaith splitting, $\Sigma^\infty\Omega^\infty\Sigma^\infty X\cong \coprod_{i=1}^{\infty}(\Sigma^\infty X)^{\wedge i}_{h\Sigma_i}$ for any connected pointed space $X$, and the equivalence is induced from the map $\Sigma^\infty X\rightarrow \Sigma^\infty(\Omega^\infty\Sigma^\infty X)_+$  by extending it into an $\mathcal{E}_\infty$ map. See \cite{kuhn06} for details.

Fix a prime $p$. The Jame-Hopf map $$jh: \Omega^\infty\Sigma^\infty X\rightarrow \Omega^\infty (\Sigma^\infty X)^{\wedge p}_{h\Sigma_p}$$ is the adjoint of the projection map $$\Sigma^\infty\Omega^\infty\Sigma^\infty X\rightarrow (\Sigma^\infty X)^{\wedge p}_{h\Sigma_p}$$  
By definition the infinite suspension of the James-Hopf map lands in the category of weakly graded spectra (see section \ref{ocv} for the definition). We will study the effect of the Bousfield-Kuhn functor on this map, $$\Phi_{K(n)}jh:L_{K(n)}\Sigma^\infty X\rightarrow L_{K(n)}(\Sigma^\infty X)^{\wedge p}_{h\Sigma_p}$$ 

By results of \cite{br12}, we can identify $\Phi_{K(n)}X$ with $$TAQ_{L_{K(n)}\mathbb{S}}(L_{K(n)}D\Sigma^\infty\Omega^\infty\Sigma^\infty X_+)$$ 
the topological Andr\'{e}-Quillen cohomology of the $K(n)$-local $\mathcal{E}_\infty$ spectrum $L_{K(n)}D\Sigma^\infty\Omega^\infty\Sigma^\infty X_+$ where the multiplication is induced from the diagonal on $\Omega^\infty\Sigma^\infty X$.

In order to apply the theorem, we need to understand the $\mathcal{E}_\infty$ structure on $D\Sigma^\infty\Omega^\infty\Sigma^\infty X_+$ induced from the diagonal map on $\Omega^\infty\Sigma^\infty X$. The diagonal map is determined by \cite{kuhn00}. Essentially, the diagonal map on the component $\Sigma^\infty X$ in the Snaith splitting is the sum of two maps: the primitive part $X\rightarrow (\mathbb{S}\wedge X)\coprod (X\wedge \mathbb{S})$, and the other part  $X\rightarrow X\wedge X$ coming from the diagonal map on $X$.

We will work at a slightly more general setting. To simplify the argument, we will work in the category of $S$-modules in the rest of this section, so that we have a symmetric monoidal smash product, and for cofibrant $X$, the action of $\Sigma_k$ on $X^{\wedge k}$ is free.
We will also work in the $K(n)$-local category, and $K(n)$-localization will be applied implicitly when necessary.

Let $X$ be a cocommutative $S$-coalgebra with structure map $\psi_X:X\rightarrow X\wedge X$. Suppose that $X$ is cofibrant and $E_n^*(X)$ is free of finite rank over $E_n^*$. Hence in the $K(n)$-local category, $X$ is dualizible by \cite{hs99}. We can construct the free commutative $S$-algebra $S(X)=\oplus_k X^{\wedge k}/\Sigma_k$ generated by $X$. We have a cocommutative $S$-coalgebra structure on $S(X)$ defined by extending the composition of the following map 
$$X\xrightarrow{1\wedge id+id\wedge 1 +\psi_X}(\mathbb{S}\wedge X)\oplus (X\wedge \mathbb{S})\oplus (X\wedge X)\rightarrow S(X)\wedge S(X)$$
into a map of commutative $S$-algebras. (The last map is the inclusion of summand $(\mathbb{S}\wedge X)\oplus (X\wedge \mathbb{S})\oplus (X\wedge X)\rightarrow\oplus_{i,j} (X^{\wedge i})\wedge (X^{\wedge j})/\Sigma_i\times\Sigma_j$.)

Note that in this case, we have a natural decreasing filtration on $S(X)$ defined by the powers of $X$, which is preserved by the coproduct defined above. Taking the restricted Spanier-Whitehead dual, we have the spectrum $D'S(X)=\oplus_kD(X^{\wedge k}/\Sigma_k)$, which is $K(n)$-locally equivalent to $\oplus_k((DX)^{\wedge k})^{\Sigma_k}$ because $X$ is $K(n)$-locally dualizibile. The coproduct above induces a commutative $S$-algebra structure on $D'S(X)$, since the formula for the product is a finite sum on each summand. Moreover, the product preserves the increasing filtration defined by the powers of $DX$.

If we take the graded pieces of the $S$-algebra $D'S(X)$, we get a graded commutative $S$-algebra, and the following lemma is straightforward:
\begin{lem}
Suppose we have the graded commutative $S$-algebra structure over $D'S(X)$ defined by the coproduct induced from $$X\xrightarrow{1\wedge id+id\wedge 1 }(\mathbb{S}\wedge X)\oplus (X\wedge \mathbb{S})\rightarrow S(X)\wedge S(X)$$ by extending it into an $S$-algebra map. Then
the restriction of the iterated product map $(D'S(X))^{\wedge k}/\Sigma_k\rightarrow D'S(X)$ to the summand $DX^{\wedge k}/\Sigma_k$ is the norm map to the summand $(DX^{\wedge k})^{\Sigma_k}$: $$DX^{\wedge k}/\Sigma_k\rightarrow (DX^{\wedge k})^{\Sigma_k}\rightarrow D'S(X)$$
\end{lem}

\begin{cor}\label{dxf}
In the $K(n)$-local category, the inclusion of $DX$ in $D'S(X)$ induces an equivalence of the free commutative $S$-algebra generated by $DX$, into $D'S(X)$.
\end{cor}
\begin{proof}
It is enough to check the graded pieces, since the induced map $S(DX)\rightarrow D'S(X)$ preserves the natural filtrations. 

If we look at the corresponding graded commutative $S$-algebra for $D'S(X)$, we find that the product is as the one in the lemma. Then the lemma shows that the induced map $S(DX)\rightarrow D'S(X)$ is the norm map on each homogeneous component. Then the corollary follows.
\end{proof}

Using this corollary, we can identify the topological Andr\'{e}-Quillen cohomolgoy of $D'S(X)$ with $DX$.

\begin{lem}
Suppose further that $X$ is the suspension spectrum of a space $Y$ with coproduct induced form the diagonal map.
Then the effect of the Bousfield-Kuhn functor on the James-Hopf map for $Y$ is the dual of the composite $$(DX^{\wedge p})^{\Sigma_p}\rightarrow D'S(X)\xleftarrow{\phi} S(DX)\rightarrow DX$$ 
where the map $\phi$ is the $K(n)$-equivalence in corollary \ref{dxf}, and the last map is the projection.
\end{lem}
\begin{proof}
Behrens and Rezk constructed in Section 6 of \cite{br12} a natural transformation $$TAQ(D\Sigma^\infty Z_+)\rightarrow D\Phi_n Z$$
Now when $Z$ takes the form $\Omega^\infty\Sigma^\infty Y$, we have a natural transformation $TAQ(D'\Sigma^\infty Z_+)\rightarrow TAQ(D\Sigma^\infty Z_+)$ when we regard $Z$ as an object in the category of weak infinite loop spaces labeled by $Y$. The composition gives a natural transformation $TAQ(D'\Sigma^\infty Z_+)\rightarrow D\Phi_n Z$ from the category of weak infinite loop spaces to the category of spectra. Now it is easy to see that, the following natural transformation $$TAQ(D'\Sigma^\infty\Omega^\infty\Sigma^\infty Y_+)\rightarrow D\Phi_n\Omega^\infty\Sigma^\infty Y$$ is an equivalence since both sides can be identified with $D\Sigma^\infty Y$. So we conclude that the natural transformation $TAQ(D'\Sigma^\infty Z_+)\rightarrow D\Phi_n Z$ is an equivalence. The lemma follows from the previous corollary using this TAQ description of the Bousfield-Kuhn functor on weak infinite loop spaces.
\end{proof}

Now we will investigate the maps $$j_i: (DX^{\wedge i})^{\Sigma_i}\rightarrow D'S(X)\xleftarrow{\phi} S(DX)\rightarrow DX$$ for $i\leq p$. Of course $j_1$ is the identity map, and what we want to do is to find how the direct summand in $S(DX)$ transport to $D'S(X)$.
\begin{lem}
For $i<p$, the map $j_i$ factors through the multiplication map $DX^{\wedge i}\rightarrow DX$. In fact, for all $i$, the composition of $j_i$ with the transfer map $DX^{\wedge i}\rightarrow (DX^{\wedge i})^{\Sigma_i}$ is a multiple of the multiplication map.
\end{lem}
\begin{proof}
We will show this by induction. The case for $i=1$ is clear.
So we will assume that we have the lemma for all $j_t$ for $t<i$.

Now to investigate $j_i$, we will first find the image of $DX^{\wedge i}/\Sigma_i$ under $\phi$, which goes to zero under projection.
By \cite{kuhn00}, the restriction of $\phi$ to $DX^{\wedge i}/\Sigma_i$ is the sum of the maps 
$$\phi_I:DX^{\wedge i}/\Sigma_i\rightarrow (DX^{\wedge |I|})^{\Sigma_{|I|}}$$ 
where $I$ ranges over all partitions of $i$, into $|I|$ parts. (This corresponds to the special case of a covering by $i$ copies of $1$'s in the notation of \cite{kuhn00}.) The maps $\phi_I$ are described as follows.
Let $\Sigma_I\subset \Sigma_i$ be the subgroup of permutations which preserve the (unordered) partition $I$. For example, $ \Sigma_{(p)}=\Sigma_{(1,\dots,1)}=\Sigma_i$. So if $I$ has $i_m$ components with $m$ elements, then we have $DX^{\wedge i}/\Sigma_I=\wedge_m(DX^{\wedge m}/\Sigma_{m})^{\wedge i_m}/\Sigma_{i_m}$. Then $\phi_I$ is defined to be the composition 
\begin{eqnarray*}
DX^{\wedge i}/\Sigma_i\rightarrow DX^{\wedge i}/\Sigma_I=\wedge_m(DX^{\wedge m}/\Sigma_{m})^{\wedge i_m}/\Sigma_{i_m}\rightarrow
\\ \wedge_mDX^{\wedge i_m}/\Sigma_{i_m}\rightarrow DX^{\wedge |I|}/\Sigma_{|I|}\rightarrow (DX^{\wedge |I|})^{\Sigma_{|I|}}
\end{eqnarray*} 
Here the first map is the transfer, the second induced from the multiplication defined by $\psi$, the third map is the restriction, and the last map is the norm map.

Now we have the equation $\sum_I j_{|I|}\circ \phi_I=0$. We compose this equation with the restriction map $DX^{\wedge i}\rightarrow DX^{\wedge i}/\Sigma_i$. Then the summand with $|I|=(1,\dots,1)$ becomes the composition of $j_i$ with the transfer map, since the composition of the norm map with the restriction is the transfer map. For the other summands, we find that the composition of the restriction map with the map 
\begin{eqnarray*}
DX^{\wedge i}/\Sigma_i\rightarrow DX^{\wedge i}/\Sigma_I=\wedge_m(DX^{\wedge m}/\Sigma_{m})^{\wedge i_m}/\Sigma_{i_m}\rightarrow \\ \wedge_mDX^{\wedge i_m}/\Sigma_{i_m}\rightarrow DX^{\wedge |I|}/\Sigma_{|I|}
\end{eqnarray*} 
is the sum of maps $$DX^{\wedge i}\xrightarrow{g} \wedge_mDX^{\wedge i_m}\rightarrow DX^{\wedge |I|}\rightarrow DX^{\wedge |I|}/\Sigma_{|I|}$$
where the first map rearranges the order of factors according to $g$ for $g$ running through $\Sigma_i/\Sigma_{|I|}$. The second map is the multiplication map, and the last map is the restriction map. So the composition of the map $j_{|I|}\circ \phi_I$ with the transfer map $DX^{\wedge i}\rightarrow (DX^{\wedge i})^{\Sigma_i}$ becomes the sum of compositions $$DX^{\wedge i}\rightarrow DX^{\wedge |I|}\rightarrow (DX^{\wedge |I|})^{\Sigma_I}\xrightarrow{j_{|I|}} DX$$
where the first map is the multiplication map of type $I$ described above, and the second map is the transfer. The sum is over the set $\Sigma_i/\Sigma_I$ of ways to do the multiplication. 

So by induction, we proved that the composition of $j_i$ with the restriction $DX^{\wedge i}\rightarrow (DX^{\wedge i})^{\Sigma_i}$ is a multiple of the multiplication map. (In fact, we know the coefficient is the alternating sum of ways to do the multiplication, but we do not need this formula.) Now in case $i<p$, the restriction map is projection to a direct summand, so we can find a section, and the lemma is proved.

\end{proof}

The case $i=p$ is a bit subtle, since then the restriction map is no longer a projection. This is why it is the interesting James-Hopf map. We have the following description:
\begin{thm} \label{jhc}
The map $j_p$ is the sum of two maps. The first map is the negative of the map $(DX^{\wedge p})^{\Sigma_p}\leftarrow DX^{\wedge p}/\Sigma_p\rightarrow DX$ where the first map is the norm map and the second the multiplication map defined by the map $\psi$. The other map factors through the multiplication map $DX^{\wedge p}\rightarrow DX$.
\end{thm}
\begin{proof}
We follow the arguments of the previous lemma. We still have the equation $\sum_I j_{|I|}\circ \phi_I=0$. Now in this case, the summand corresponding to $I=(1,\dots,1)$ is the composition of $j_i$ with the norm map, and the summand corresponding to $I=(p)$ is the multiplication map. So these give the map $(DX^{\wedge p})^{\Sigma_p}\leftarrow DX^{\wedge p}/\Sigma_p\rightarrow DX$.

For the other summands, we note that $\Sigma_I$ is a proper subgroup for $I\neq (1,\dots,1), (p)$. Then the restriction map $DX^{\wedge p}\rightarrow DX^{\wedge p}/\Sigma_I$ admits a section. As before, we find that the composition $$DX^{\wedge p}\rightarrow DX^{\wedge p}/\Sigma_I \rightarrow DX^{\wedge |I|}/\Sigma_{|I|}$$ is the composition $$DX^{\wedge p}\rightarrow DX^{\wedge |I|}\rightarrow DX^{\wedge |I|}/\Sigma_{|I|}$$ 
The only difference is that we no longer need the permutation of factors. The same argument shows that all such summands $j_{|I|}\circ \phi_I$ with $|I|\neq (1,\dots,1), (p)$ factors through the multiplication map $DX^{\wedge p}\rightarrow DX$.
This completes the proof
\end{proof}

\begin{rem}
We do not know how to describe the map factoring the second map through the multiplication map. In ordinary cohomology, it look like the restriction map. However, in the case we are interested in, i.e. for $X=S^{2n+1}$, the multiplication map is homotopically zero nonequivariantly, so in this case the second map has no contributions.
\end{rem}

\section{Determination of the cohomological James-Hopf map} \label{dc}
We will show that theorem \ref{jhc} determines the effect of the James-Hopf map on $E_n$-cohomology for odd spheres. 

We need to understand the map $$(DX^{\wedge p})^{h\Sigma_p}\leftarrow (DX^{\wedge p})_{h\Sigma_p}\rightarrow DX$$

First we look at the case when $X$ is the sphere spectrum.
In this case, the multiplication map $\mathbb{S}_{h\Sigma_p}\rightarrow \mathbb{S}$ is simply the restriction map along the homomorphism $\Sigma_p\rightarrow *$ to the trivial group.

To understand the norm map, we will use the following description. For any $Y$ with an action by a finite group $G$, the identity map is $G$-equivariant, so we have a canonical element $id_G\in \pi_0\mathbb{H}om(Y,Y)^{hG}$, where $\mathbb{H}om(Y,Y)$ is the inner $Hom$, and $G$ acts on it by conjugation. The following lemma is a consequence of discussions in the appendix of \cite{br12}:
\begin{lem}
The class of the norm map $Y_{hG}\rightarrow Y^{hG}$ in $\pi_0$ of the spectrum $\mathbb{H}om(Y_{hG},Y^{hG})=\mathbb{H}om(Y,Y)^{h(G\times G)}$, is the transfer of $id_G$ along the diagonal homomorphism $G\rightarrow G\times G$.
\end{lem}

Now specialize to the case of $\mathbb{S}$. The map $id_G$ is simply the restriction map $\mathbb{S}\rightarrow \mathbb{S}^{hG}$ along the map $G\rightarrow *$.  Now if $G$ is abelien, $G$ is the pull back of the diagram $G\times G\rightarrow G\leftarrow *$, where the first map sends $(g_1,g_2)$ to $g_2^{-1}g_1$. In this case, the projection formula says the composition of restriction and transfer equals the composition of transfer and restriction: $\mathbb{S}\xrightarrow{tr} \mathbb{S}^{hG}\xrightarrow{res} \mathbb{S}^{h(G\times G)}$.

If $H$ is a subgroup of $G$, we consider the composition $$\mathbb{S}\xrightarrow{res}\mathbb{S}^{hG}\xrightarrow{tr}\mathbb{S}^{h(G\times G)}\xrightarrow{res}\mathbb{S}^{h(H\times H)}$$ of the restriction, transfer and restriction map. Using the double coset formula, the composition of the latter two maps is the sum of maps $\mathbb{S}^{hG}\rightarrow\mathbb{S}^{h(H\cap H^g)}\rightarrow\mathbb{S}^{h(H\times H)}$, where the sum is over $g\in G/H$, and the first map is restriction, the second map is the transfer along the map $H\cap H^g\rightarrow H\times H$ whose first component is the inclusion and the second component is conjugation by $g$.

Now we take $G=\mathbb{Z}/p$. In this case, $G$ is abelien. So we will look at the composition of $\mathbb{S}\xrightarrow{tr} \mathbb{S}^{hG}\xrightarrow{res}\mathbb{S}^{h(G\times G)}$ of transfer and restriction. 
After taking $E_n$-homlogy, we can identify ${E_n}_*(\mathbb{S}^{hG})=E_n^*[[\xi]]/[p](\xi)$ with the ring of functions on the finite flat algebraic group of order $p$ points in the formal group $F_n$ associated with $E_n$. 
The transfer map $E_n^*(\mathbb{S})\rightarrow E_n^*(\mathbb{S}_{hG})$ has the following formula:
\begin{lem}
The transfer of $1\in E_n^*(\mathbb{S})$ is the element $$\frac{[p](\xi)}{\xi}\in E_n^*(\mathbb{S}_{hG})=E_n^*[[\xi]]/[p](\xi)$$
\end{lem}
\begin{proof}
See page 588 in \cite{hkr}.
\end{proof}

We know that the addition on $\mathbb{Z}/p$ correspond to the multiplication of the formal group law. So the restriction map along $$G\times G\rightarrow G$$ gives the map on cohomology sending the element $$f(\xi)\in E_n^*(\mathbb{S}_{hG})=E_n^*[[\xi]]/[p](\xi)$$ to the element $$f(\xi-_{F_n}\eta)\in E_n^*(\mathbb{S}_{h(G\times G)})=E_n^*[[\xi,\eta]]/([p](\xi),[p](\eta))$$
Hence the class of the norm map has Hurewicz image $$\frac{[p](\xi-_{F_n}\eta)}{\xi-_{F_n}\eta}\in E_n^*(\mathbb{S}_{h(G\times G)})$$

Now we take some embedding $E_n^*\rightarrow K$ into some algebraically closed field of characteristic zero. Then over $K$, the subgroup of order $p$ points becomes the constant group $G\times Spec(K)$. Then the algebra of functions becomes the algebra of functions on the discrete group $G$ with values in $K$. Then the element $\frac{[p](x)}{x}$ is $p$ times the characteristic function on the unit, since its multiplication with $x$ is zero and have value $p$ when $x=0$.
Hence over $K$ the Hurewicz image of the norm map is $p$ times the characteristic function on the diagonal of $G\times G$.

Now we look at the case when $G=\Sigma_p$, which we are interested in. In this case, we take the subgroup $H=\mathbb{Z}/p\in G$.
Then the Weyl group of $H$ in $G$ is $\mathbb{F}_p^{\times}$. Recall from \ref{ebg} that $E_n^*(\mathbb{S}_{hG})$ can be identified with $\mathbb{F}_p^{\times}$-invariant elements in $E_n^*(\mathbb{S}_{hH})$ under the restriction map. 

We have shown how to compute the map $$\mathbb{S}\rightarrow\mathbb{S}^{hG}\rightarrow\mathbb{S}^{h(G\times G)}\rightarrow \mathbb{S}^{h(H\times H)}$$ using the double coset formula. There are two case of $H\cap H^g$ according to whether $g$ is contained in the normalizer of $H$ or not.
The first case gives the $p-1$ components $S\rightarrow S^{h(H\times H)}\rightarrow S^{h(H\times H)}$ indexed by elements $g$ in the Weyl group of $H$, where the first map is the class of the norm map of $H$, and the second map has $id$ on the first component, and conjugation by $g$ on the second component. The other components are $\frac{(p-1)!-(p-1)}{p}$ times the transfer map $S\rightarrow S^{h(H\times H)}$. Hence over $K$, the Hurewicz image of the norm map is the function on $H\times H$ which has value $p$ at any pair $(P,Q)\in H\times H$ if they are both not the unit and $Q$ is a multiple of $P$, and value zero otherwise unless $(P,Q)=(0,0)$ in which case we have the value $p!$.

We can take a basis of the $E_n$ cohomology of $\mathbb{S}_{hG}$ over $K$ to be the characteristic functions of the classes of $H$ under the action of $\mathbb{F}_p^{\times}$. Call them $f_0,\dots, f_q$, where $f_0$ is the characteristic function on $0\in H$.
Then the Hurewicz image of the norm map is the element $$p!f_0\otimes f_0+pf_1\otimes f_1\otimes f_1+\dots+pf_q\otimes f_q$$ 
Hence the nondegenerate quadratic form on $E_n^*(\mathbb{S}_{hG})$ defined by the norm map has the formula $$<f,g>=\frac{1}{p!}f(0)g(0)+\frac{1}{p(p-1)}\sum_{P\in H\setminus\{0\}}f(P)g(P)$$ over $K$.
Note that this formula also makes sense over $E_n^*$.

Since the restriction map $\mathbb{S}_{hG}\rightarrow \mathbb{S}$ sends $1$ to $1$ in cohomology, we arrive at the conclusion that the map $\mathbb{S}^{hG} \leftarrow \mathbb{S}_{hG}\rightarrow \mathbb{S}$ has the effect on homology sending any element $f\in E_n^*(S_{hG})$ to the element $$<f,1>=\frac{1}{p!}f(0)+\frac{1}{p(p-1)}\sum_{P\in H\setminus\{0\}}f(P)$$

Now we look at the case of $S^k$. In that case, we have a commutative diagram
\begin{equation}
\begin{array}{ccccc}
S^k\wedge \mathbb{S}_{h\Sigma_p}&\rightarrow&S^k\wedge \mathbb{S}^{h\Sigma_p}&\leftarrow& \Sigma^\infty S^k\\
\downarrow&&\downarrow&&\downarrow\\
(\Sigma^\infty S^{kp})_{h\Sigma_p}&\rightarrow&(\Sigma^\infty S^{kp})^{h\Sigma_p}&\leftarrow&\Sigma^\infty S^k
\end{array}
\end{equation}
Here the vertical maps are induced from the diagonal map on $S^k$.
Note that the rows are the dual of the maps we have just studied, since the dual of the norm map is still the norm map of the dual.

The first vertical map is the quotient $P_0\rightarrow P_k$. We observe that on $E_n$-cohomology, this map is the inclusion of the ideal generated by $(\xi^{p-1})^m$ if $k=2m-1$ is odd. So we get the same formula on cohmology of the second row under this identification.

As we have observed, for odd spheres, the diagonal map is non-equivariantly homotopic to zero, so the other nasty terms disappear. Hence we prove Theorem \ref{sfl}.

\section{The monochromatic homological Hopf invariant} \label{mh}

In this section, we will compute the monochromatic homological Hopf invariant, i.e. the behavior of the James-Hopf map on the cohomology of the second monochromatic layer: $$jh^*: {E_2}_*\Phi_{M_2}QS^0\rightarrow {E_2}_*\Phi_{M_2}QB\Sigma_p$$

We will start with a general $E_n$. 
Recall that ${E_n}_*$ is the complete local ring (in the graded sense) $\mathbb{Z}_p[[u_1,\dots,u_{n-1}]][u^{\pm}]$ with maximal ideal $I=(p,\dots,u_{n-1})$.
Let $$M={E_n}_*/(p^\infty,\dots,v_{n-1}^\infty)$$
Following \cite{lch}, we know that, as a graded ring, ${E_n}_*$ is Gorenstein and $M$ is the dualizing module for the functor $Ext^n(-,{E_n}_*)$ in the category of graded ${E_n}_*$ modules.

Recall that ${E_n}_*M_n\mathbb{S}=M[-n]$, where $M[-n]$ means the module $M$ with a shift in grading.
In general, we have $M_nX=\varinjlim W_k\wedge X$, where $W_k$ is a system of type $n$ complexes of the form $S/(p^{i_0},\dots,v_{n-1}^{i_{n-1}})$ with the numbers $i_0,\dots,i_{n-1}$ tending to infinity. When $E_n^*X$ is finitely generated, ${E_n}_*W_k\wedge X$ are all finite length ${E_n}_*$ modules. Since ${E_n}_*$ is a Gorenstein grade ring, the functor $Ext^s(-,{E_n}_*)$, restricted to the category of finite length ${E_n}_*$ modules, is zero for $s\neq n$ and equals $Hom(-,M)$ when $s=n$. Hence by the universal coefficient theorem, $E_n^*W_k\wedge X=Hom({E_n}_*W_k\wedge X,M)[n]$. Taking the limit, we conclude $$E_n^*X=Hom({E_n}_*M_nX,M)[n]$$, since the $\varprojlim^1$ term vanishes when we assume $E_n^*X$ is finitely generated.
Dually, we have $${E_n}_*M_nX=Hom(E_n^*X,M)[-n]$$

In the case of $P_1=\Sigma^\infty B\Sigma_p$, we have the following description. Recall that $E_n^*P_1$ is the $\mathbb{F}_p^{\times}$-invariants of the ideal generated by $\xi$ in the ring $E_n^*[[\xi]]/[p](\xi)$. Since we are working with a $p$-typical formal group law, the polynomial $[p](\xi)$ equals $\xi q(\xi^{p-1})$ for some polynomial $$q(y)=p+q_1y+q_2y^2+\dots$$
So we have:
\begin{lem} 
$E_n^*P_1$ is generated by the elements $y,y^2,\dots$, subject to the relation $py^i+q_1y^{i+1}+\dots=0$, for $i\geq1$. 
\end{lem} 
The filtration defined by the maps $$P_1\rightarrow P_3\rightarrow P_5\rightarrow\dots$$ corresponds to the decreasing filtration defined by the powers of $y$ in cohomology. Dually, we have the corresponding increasing dual filtration on ${E_n}_*M_n P_1=Hom(E_n^*P_1,M)[n]$.

Then the monochromatic homological Hopf invariant problem is to find, for any element $\mu$ of ${E_2}_*M_2\mathbb{S}$, the lowest filtration where $jh_*(\mu)$ lies, and its image to the corresponding sub-quotient.

The duality between homology and cohomology gives a perfect pairing $E_n^*X\otimes{E_n}_*M_nX[n]\rightarrow M$. 
Using this pairing, we can state the Hopf invariant problem dually:
for any element $$\frac{1}{p^{i_0}v_1^{i_1}\dots v_{n-1}^{i_{n-1}}}\in M$$ find the largest number $k$ for which $$<\frac{1}{p^{i_0}v_1^{i_1}\dots v_{n-1}^{i_{n-1}}}, jh^*(y^k)>\neq 0$$ and find the value of the pairing.

\begin{rem}
Here we only study the problem of pairing with monomials. This suffices to detect the Hopf invariant of the $\beta$ family. In fact, we can give a natural partial ordering on the monomials defined by divisibility relations. In the monochromatic expression for the $\beta$ family elements, there is always a unique leading term under this partial ordering, and one can easily see that the non-vanishing of pairings with any element in cohomolgy is determined only by the leading term.
\end{rem}

So we need to understand the elements $jh^*(y^k)$. From Theorem \ref{sfl}, this amounts to compute the sum of the powers of the roots of the power series $q(y)$. We will use Weierstrass preparation theorem to transform it into a polynomial and use Newton's identities.

From now on we will work with chromatic level two. So we will set $n=2$. In the remainder of this section, we will also assume $p$ to be an odd prime. The case $p=2$ will not be discussed in this section.

To begin with, recall that $$[p](x)=px+_{F_2}v_1x^p+_{F_2}v_2x^{p^2}$$ with the Araki generators.
Let $h(x)=(px+v_1x^p)+_{F_2}v_2x^{p^2}$. Then it is easy to see that $h(x)\equiv [p](x) \mod (pv_1)$.

Because we are working with a p-typical formal group law, and since $p$ is odd, we have $[-1](x)=-x$.
Hence the Weierstrass factor of $h(x)$ is $px+v_1x^p+v_2x^{p^2}$, since they have the same roots. From the uniqueness of Weierstrass preparation theorem, we conclude that this is the Weierstrass factor for $[p](x)$ modulo $(pv_1)$.

Hence we find that the Weisstrass factor $$q_0(y)=y^{p+1}+c_1y^p+\dots+c_{p+1}$$ for $q(y)$ has the property that, $c_{p+1}$ is $p$ times a unit, and $c_p$ is $v_1$ times a unit, and the other $c_i$'s are divisible by $pv_1$.

Our goal is to find the power sums of the roots of $q_0(y)$.

Let $e_i$ be the elementary symmetric polynomials of the roots of $q_0(y)$. Then $e_i=\pm c_i$. 

The following lemma is a well known consequence of Newton's identity:
\begin{lem}
We have the following formula:
$$s_k=\sum\frac{k(\sum_i r_i-1)!}{\prod_i r_i!}\prod_i(-e_i)^{r_i}$$
where $s_k$ denotes the sum of $k^{th}$ powers, and the sum in the right hand side is over numbers $r_i$ with $\sum_i ir_i=k$.
\end{lem}

To address the question of pairing with $\frac{1}{p^{s+1}v_1^t}$, we will first study the pairing on the individual terms.

\begin{lem}
For fixed $s,t$,
among all the terms $$\frac{1}{p}\frac{k(\sum_i r_i-1)!}{\prod_i r_i!}\prod_i c_i^{r_i}$$ 
with $k=\sum ir_i$, (recall that the factor $\frac{1}{p}$ comes from the formula of the James-Hopf map), which have nontrivial pairing with $\frac{1}{p^{s+1}v_1^t}$, the one with $k=\sum_iir_i$ largest, has the following property: whenever $i<p$, we have $r_i=0$.
\end{lem}
\begin{proof}
Suppose the contrary, we will construct a new term as follows.
If some $r_{i_0}\neq0$ with $i_0<p$. Set $l=p\;or\;p+1$. Then we construct the new term with $r'_l=r_l+r_{i_0}$, $r'_{i_0}=0$ and the other $r'_i=r_i$ for $i\neq i_0,l$. Then $\frac{1}{p}\frac{(\sum_i r_i-1)!}{\prod_i r_i!}\prod_i c_i^{r_i}$ is divisible by $\frac{1}{p}\frac{(\sum_i r'_i-1)!}{\prod_i r'_i!}\prod_i c_i^{r'_i}$, and $k'>k$.
So if $k'$ is divisible by no larger powers of $p$ than $k$ do, then we are done.

Note that we have two choices of $l$. We will show that at least one choice has the desired property. Let $k_1$ and $k_2$ be the number $\sum_iir'_i$ for the choice $l=p$ and $l=p+1$ respectively. Then we find $k_1-k=(p-i_0)r_{i_0}$, and $k_2-k_1=r_{i_0}$.
Thus the following easy lemma completes the proof.
\end{proof}
\begin{lem}
Let $a,b,c$ be integers. Suppose $\mathfrak{v}_p(b-a)\geq \mathfrak{v}_p(c-b)$, where $\mathfrak{v}_p$ is the $p$-adic valuation. Then we cannot have both $\mathfrak{v}_p(b)>\mathfrak{v}_p(a)$ and $\mathfrak{v}_p(c)>\mathfrak{v}_p(a)$.
\end{lem}

Recall we want to compute terms of the form
$$<\frac{1}{p^{i+1}v_1^j},\frac{1}{p}s_k>$$ where $s_k$ is the power sum of the roots of $f(y)$.
We will study the pairing with terms in the previous lemma, i.e. $$\frac{1}{p}\frac{k(s+t-1)!}{s!t!}p^sv_1^t$$ which contributes to $s_k$ with $k=(p+1)s+pt$.

For the study of $\beta$ family at odd primes, we will only consider the case when $j$ is divisible by $p^i$.

For the pairing to be nontrivial, we must have $t\leq j-1$. We also need to consider divisibility for $p$.
Since $\frac{k(s+t-1)!}{s!t!}$ is always an integer, a necessary condition for the pairing to be nontrivial is $s\leq i+1$.

When $i=0$, then $s\leq1$, and it is not hard to find that $(s,t)=(1,j-1)$ is the unique choice for the largest $m$, and in this case $k=pj+1$ is not divisible by $p$. Hence we conclude that $jh_*(\frac{1}{pv_1^j})$ lies in filtration $pj+1$, and its image in the corresponding subquotient is $\frac{1}{v_1}$ up to units.

Now we will assume $i\geq 1$. Then $j$ is divisible by $p$. We also note that $s\leq p^i$ in this case. 

First we need a lemma on binomial coefficients:
\begin{lem}
The largest power of $p$ dividing $\frac{s!}{r!(s-r)!}$ does not exceed $\frac{\log s}{\log p}$.
\end{lem}
\begin{proof}
That largest power is $([\frac{s}{p}]-[\frac{r}{p}]-[\frac{s-r}{p}])+([\frac{s}{p^2}]-[\frac{r}{p^2}]-[\frac{s-r}{p^2}])+\dots$, and each term is at most $1$.
\end{proof}

Now we will study which terms of the form $\frac{1}{p}\frac{k(s+t-1)!}{s!t!}p^sv_1^t$ has nontrivial pairing with $\frac{1}{p^{i+1}v_1^j}$.

There are two cases to consider. First assume that $s+t>j$. In that case, we observe that $\frac{(j-1)!}{t!(j-1-t)!}$ is a $p$-adic unit, since $$j-1-t<s+t-1-t=s-1<p^i$$ and $j$ is divisible by $p^i$.
The remaining factors give $$\frac{1}{p}\frac{kj(j+1)\dots(s+t-1)}{(j-t)\dots(s-1)s}p^sv_1^t$$
Since $t<j$, $s\leq p^i$, we have $s+t-1-j<p^i$. Since $p^i$ divides $j$, the factor $(j+1)\dots(s+t-1)$ has the same $p$-valuation as $(s+t-j-1)!$. By the lemma, $\frac{(s-2)!}{(j-t-1)!(s+t-j-1)!}$ has $p$-valuation at most $\frac{\log(s-2)}{\log p}$, (we understand this number to be $0$ when $s=2$). We also note that at most one of $s$ and $s-1$ is divisible by $p$. Hence the $p$-valuation of $$\frac{1}{p}\frac{j(j+1)\dots(s+t-1)}{(j-t)\dots(s-1)s}p^s$$ is at least $-1+i-\frac{\log(s-2)}{\log p}-\frac{\log s}{\log p}+s$. So we need $$s\leq 1+\frac{\log s(s-2)}{\log p}$$ for this term have nontrivial pairing with $\frac{1}{p^{i+1}v_1^j}$. So the only choices are $p=2$ and $s=2,4$. But in those cases, it is not hard to find that $k$ has extra factors of $2$ and we actually get zero pairing.

Hence the case $s+t>j$ is completely excluded, and we must have $s+t\leq j$. Since $s\leq i+1$, the largest possible $k$ is obtained from the choice $s=i+1$, and $t=j-i-1$. Then $k=pj+i+1$. Then $\frac{k}{s}$ and $\frac{(s+t-1)!}{(s-1)!t!}$ are both $p$-adic units, and the term is the unique one with nontrivial pairing with $\frac{1}{p^{i+1}v_1^j}$ and largest $k$. One checks that the pairing is $\frac{1}{pv_1^{i+1}}$ up to a unit.

This proves Theorem \ref{hlhi}, when one traces the degree to get the actual powers of $v_2$.

\section{Comparison with the modular isogeny complex} \label{cmp}

Behrens and Rezk give another description of the $E_n$-homology of the unstable spheres using the modular isogeny complex. See \cite{br12} and \cite{rez13} for details. We will show that this description is compatible with our formula in the special case of chromatic height two.

First note that we are computing using cohomology and the modular isogeny complex is in terms of homology. So we need the self-duality of $L_{K(n)}\Sigma^\infty {B\Sigma_p}_+$ to do the comparison.

Let $R=E_n^*[x]/q(x)$. Recall $\xi q(\xi^{p-1})=[p](\xi)$. Then there is the trace map $tr:R\rightarrow E_n^*$. We define a pairing $<,>:R\times xR\rightarrow E_n^*$ by the formula $<\lambda,\mu>=\frac{tr(\lambda\mu)}{p}$. One can see that this is a perfect pairing since $<x^i,x^{\frac{p^n-1}{p-1}-i}>$ is a unit and $<x^i,x^j>$ lies in the maximal ideal when $i+j<\frac{p^n-1}{p-1}$. Using this pairing, we can identify ${E_n}_*^{\wedge}L_{K(n)}P_i$ with $x^{1-i}R$, where $x^{1-i}R$ denotes the sub-$R$-module of $R[x^{-1}]$ generated by $x^{1-i}$. Then the dual of the effect of the James-Hopf map on cohomology becomes simply the inclusion of the unit $E_n^*\rightarrow R\rightarrow x^{1-i}R$, which give the formula for the homological James-Hopf map. 

This is different from the modular isogeny complex, for we use the right unit instead of the left unit in that case. However, in the special case of $n=2$, we have an automorphism on $R$, which represents the operation of transforming the pair $(\mathbb{G},H)$ into $(\mathbb{G}/H, \mathbb{G}[p]/H)$, where $\mathbb{G}$ is a deformation of a height $2$ formal group and $H$ is an order $p$ subgroup of $\mathbb{G}$, and $\mathbb{G}[p]$ denotes of subgroup of elements of exponent $p$, which has order $p^2$. This automorphism swaps the left and right unit, so we find in this case our result agrees with the computations in \cite{rez13}. In fact, the pairing we use, after modified by the above automorphism, is the well-known Weil pairing of elliptic curves.

\section{The monochromatic stable Hopf invariant for $\beta$-families} \label{hb}

The previous section computes the behavior of the James-Hopf map on second monochromatic homology.
Then we have the corresponding map on ANSS. In this section we will show that for $p\geq5$, there are no obstructions to get a homotopical conclusion.

Recall that on the second monochromatic layer of the sphere, the $\beta_{k/j,i}$ is represented by an expression with leading term $\frac{v_2^k}{p^iv_1^j}$ on the zero line of the $E_2$-term of the ANSS.
Theorem \ref{hlhi} shows that the image of this term in ANSS lies in filtration $pj+i$, and represented by $\frac{v_2^{k-j}}{v_1^i}$ in the AHSS for the $E_2$ term of ANSS for $P_1=\Sigma^\infty B\Sigma_p$.

So there are two possibilities. One is that the James-Hopf map sends $\beta_{k/j,i}$ to the element represented by $\frac{v_2^{k-j}}{v_1^i}$ in filtration $pj+i$, and in this case we conclude that the $K(2)$-local stable Hopf invariant of $\beta_{k/j,i}$ is $\beta_{k-j/i}$. The other possibility is that the James-Hopf map sends $\beta_{k/j,i}$ to something with larger filtration. In that case, since on $E_2$-term of ANSS, the image of $\frac{v_2^{k-j}}{v_1^i}$ on the corresponding subquotient is zero because it lies in a lower filtration, we conclude that the image of $\beta_{k/j,i}$ under the James-Hopf map must be represented by terms with homological filtration at least one in ANSS. The following lemma shows that the latter does not happen if $p\geq5$.

\begin{lem}
Let $p\geq5$.
In dimension $2(p-1)k-1$, the $E_2$-term of the ANSS for $M_2P_1$ vanishes, and in dimension $2(p-1)k-2$, the terms in the $E_2$-term of ANSS all lie in homological degree $0$, with AHSS name $\frac{v_2^a}{v_1^b}$ for certain choices of $a,b$.
\end{lem}
\begin{proof}
We will use AHSS to compute the $E_2$ term of the ANSS for $M_2P_1$.

Since there is a filtration on $P_1$, with subquotients $\Sigma^{2(p-1)m-1}M(p)$, (with $M(p)$ the Moore complex), the $E_1$-term of the AHSS are the $E_2$-term of the monochromatic layer of the Moore complex.

We will use the data for the Moore complex in \cite{bk2}. First we start with the Smith-Toda complex $V(1)$.
We find that the only terms in dimensions equivalent to $0,1$ modulo $2(p-1)$ in the $E_2$-term of the ANSS for $M_2V(1)$ are the terms $v_2^a$. Hence, using the fact that $|v_1|=2(p-1)$, we conclude that on dimensions equivalent to $-1,0$ modulo $2(p-1)$, the only terms in the $E_2$-term of the ANSS of $M_2V(0)$ are the terms $\frac{v_2^a}{v_1^b}$. Hence in the $E_1$-term of the AHSS for the $E_2$-term of ANSS of $M_2P_1$, the only terms in the relevant dimensions are as claimed. 
\end{proof}

So we find that, when $p\geq5$, there are no obstructions coming from higher ANSS differentials, because the $\beta$ family lies in dimensions equivalent to $-2$ modulo $2(p-1)$, and there are no terms with higher homological filtration in the ANSS of $M_2P_1$. Hence Theorem \ref{k2hi} follows from Theorem \ref{hlhi}.

\begin{rem}
In fact, for $p\geq5$, the ANSS for $L_{K(2)}P_1$ collapses at the $E_2$-term. So all the AHSS differentials come from ANSS $d_1$-differentials, which are algebraic.
\end{rem} 

\section{Passage to the actual Hopf invariant} \label{ah}

We have given a method to understand the monochromatic Hopf invariant. The natural question is, how can we get information on the actual Hopf invariant. Here is an example:
\begin{ex}
We have shown that, for $p\geq5$, the monochromatic Hopf invariant of $\beta_1$ is $\beta_0$, which is represented by $\frac{1}{pv_1}$. This element lies on the $(p+1)$-st Moore complex. We know that it is killed in the chromatic spectral sequence by the element $\frac{v_1^{-1}}{p}$ in chromatic height $1$. From \cite{mt92}, this supports an AHSS differential to $\frac{h_0}{p}$ on the $(p-1)$-st Moore complex. So we conclude that the actual stable Hopf invariant of $\beta_1$ is $\alpha_1$, which is well-known.
\end{ex}

So we see that we can combine our knowledge on the chromatic spectral sequence to understand the actual Hopf invariant.
However, in order for this kind of argument to work, we still need an unstable analogue of the chromatic spectral sequence.

We can begin the unstable analogue as follows. Start with a topological space $X$. (Localize it at $p$ first.) We invert $p$ to get a map $X\rightarrow p^{-1}X$. We take $N_0$ to be the fiber of this map. Then we invert $v_1$ to get $N_0\rightarrow \Omega^\infty \Phi_1N_0$. Then we take $N_1$ to be its fiber. And then we invert $v_2$ on $N_1$, and so on.
This is precisely the same procedure for the stable case if we start with an infinite loop space. The $K(n)$-local version is completely analogues: we just replace the functor $\Phi_n$ by $\Phi_{M(n)}$. When the stable chromatic tower converges, we could conclude that we can recover the actual stable Hopf invariant using the monochromatic information.

However, even with such an unstable chromatic spectral sequence, there are still ambiguities. This comes from the fact that the filtrations in the chromatic spectral sequence and the Atiyah-Hirzebruch spectral sequence are different. So there is possibility that the same element might have different choice among its terms as the leading term. The effect is that, the above calculated candidate is either the actual Hopf invariant, or the actual one has a higher chromatic height, living on a higher cell. So we only have a lower bound on the sphere of origin. In favorable cases, we have constructions to achieve the lower bound, and then we can say the Hopf invariant is the computed value.

\begin{ex}
Let $p\geq5$. We know that $\beta_1$ can be desuspended to $S^{2p-1}$. We suspend it to the sphere $S^{2p+3}$. This is in the $p$-primary stable range. It is not hard to see that we can extend it to a map $V(1)\rightarrow S^{2p+3}$ in which $V(1)$ is a desuspension of the Smith-Toda complex with bottom cell in dimension $2p^2+1$, and the restriction of the map to the bottom cell is $\beta_1$. Since this is already in the stable range, the map $v_2$ desuspends, so we can construct maps $v_2^k:\Sigma^{k|v_2|}V(1)\rightarrow V(1)$. Composing the maps $\beta_1$ and $v_2^k$, we conclude that $\beta_k$ desuspends to $S^{2p+3}$ and the stable Hopf invariant is $\beta_{k-1}$, for $k\geq2$. This agrees with \cite{ben86}.

Note that how the above arguments fail for lower spheres. In the case of $S^{2p-1}$, $2\beta_1$ is not zero. This element, with Hopf invariant $\alpha_2$, is killed in the EHP sequence by the element $id$ on $S^{2p}$. In the case of $S^{2p+1}$, the Toda bracket $<\alpha_1,2,\beta_1>$ does not vanish, with Hopf invariant $\alpha_1$. This element is killed in the EHP sequence by the element $id$ on $S^{2p+2}$. Similarly on $S^{2p+2}$, we have the Toda bracket $<2,\alpha_1,2,\beta_1>$ not zero with Hopf invariant $2$. So in either cases, the element $\beta_1$ cannot be extended to $V(1)$.
\end{ex}

\end{document}